\documentclass{m2an}
\usepackage[latin9]{inputenc}
\usepackage{amsmath}
\usepackage{amssymb}
\usepackage{amsfonts}
\usepackage{a4wide}
\usepackage{graphicx}
\usepackage{nicefrac}

\newtheorem{theorem}{Theorem}[section]

\newtheorem{definition}[theorem]{Definition}

\newtheorem{lemma}[theorem]{Lemma}

\newtheorem{remark}[theorem]{Remark}

\newcommand{\Th}{\mathcal{T}_h}
\numberwithin{equation}{section}

\begin{document}
\title{An easily computable error estimator in space and time for the wave equation}
\author{O. Gorynina}\address{Laboratoire de Math\'{e}matiques de Besan\c{c}on, CNRS UMR 6623, Univ. Bourgogne Franche-Comt\'{e}, 16 route de Gray, 25030 Besan\c{c}on Cedex, France, \email{olga.gorynina@univ-fcomte.fr \ \&\ alexei.lozinski@univ-fcomte.fr}}
		  	
\author{A. Lozinski}\sameaddress{1}
            
\author{M. Picasso}\address{Institute of Mathematics, Ecole Polytechnique F\'{e}d\'{e}rale de Lausanne, Station 8, CH 1015, Lausanne, Switzerland, \email{marco.picasso@epfl.ch}}
\date{...}

\begin{abstract}
We propose a cheaper version of \textit{a posteriori} error estimator from \cite{our1} for the linear second-order wave equation discretized by the Newmark scheme in time and by the finite element method in space. The new estimator preserves all the properties of the previous one (reliability, optimality on smooth solutions and quasi-uniform meshes) but no longer requires an extra computation of the Laplacian of the discrete solution on each time step. 
\end{abstract}
%
%
\maketitle
\section*{Introduction}
\label{intro}

In this paper, we are interested in \textit{a posteriori} time-space error estimates for finite element discretisations of the wave equation. Such estimates were designed, for instance, in \cite{BS} for the case of implicit Euler discretization in time, in \cite{Georgoulis13}, \cite{Georgoulis16} for the case of second order discretizations in time by Cosine (or, equivalently, Newmark) scheme, and in \cite{our1} for a particular variant of the Newmark scheme ($\beta=1/4 $, $\gamma=1/2$). In both \cite{Georgoulis16} and  \cite{our1}, the error is measured in a physically natural norm: $H^1$ in space, $L^\infty$ in time. Another common feature of these two papers is that the  time error estimators proposed there contain the Laplacian of the discrete solution which should be computed via auxiliary finite element problems at each time step. This requires thus a non-negligible extra work in comparison with computing the discrete solution itself.  In the present paper, we propose an alternative time error estimator for the particular Newmark scheme considered in \cite{our1} that avoids these additional computations. 

In deriving our \textit{a posteriori} estimates, we follow first the approach of \cite{our1}. First of all, we recognize that the Newmark method can be reinterpreted as the Crank-Nicolson discretization of the reformulation of the governing equation as the first-order system, as in \cite{Baker76}. We then use the techniques stemming from \textit{a posteriori} error analysis for the Crank-Nicolson discretization of the heat equation in \cite{LPP}, based  on a piecewise quadratic polynomial in time reconstruction of the numerical solution. Finally, in a departure from \cite{our1}, we replace the second derivatives in space (Laplacian of the discrete solution) in the error estimate with the forth derivatives in time by reusing the governing equation. This leads to the new \textit{a posteriori} error estimate in time and also allows us to easily  recover the error estimates in space that turn out to be the same as those of \cite{our1}. The resulting estimate is referred to as the $5$-point estimator since it contains the fourth order finite differences in time and thus involves the discrete solution at $5$ points in time at each time step. On the other hand, the estimate \cite{our1} involves only 3 points in time at each time step and will be thus referred to as the $3$-point estimator. 

Like in the case of the 3-point estimator, we are able to prove that the new 5-point estimator is reliable on general regular meshes in space and non-uniform meshes in time (with constants depending on the regularity of meshes in both space and time). Moreover, the 5-point estimator is proved to be of optimal order at least on sufficiently smooth solutions, quasi-uniform meshes in space and uniform meshes in time, again reproducing the results known for the 3-point estimator. Numerical experiments demonstrate that 3-point and 5-point error estimators produce very similar results in the majority of test cases. Both turn out to be  of optimal order in space and time, even in situations not accessible to the current theory (non quasi-uniform meshes, not constant time steps). It should be therefore possible to use the new estimator  for mesh adaptation in space and time. In fact, the best strategy in practice may be to combine both estimators to take benefit from the strengths of each of them: the relative cheapness of the 5-point one, and the better numerical behavior of the 3-point estimator under abrupt changes of the mesh. 

The outline of the paper is as follows. We present the governing equations and the discretization in Section \ref{section2}. Since our work based on techniques from \cite{our1}, Section \ref{sectionRecall} is devoted to reminding the \textit{a posteriori} bounds in time and space from there. In Section \ref{section3}, the 5-point \textit{a posteriori} error estimator for the fully discrete wave problem is derived. Numerical experiments on several test cases are presented in Section \ref{section4}. 

\section{The Newmark scheme for the wave equation}\label{section2}

We consider initial-boundary-value problem for the wave equation. Let $\Omega$ be a bounded domain in $\mathbb{R}^2$ with boundary $\partial \Omega$ and $T>0$ be a given final time. Let $u=u(x,t) : \Omega\times\left[0,T \right]\to\mathbb{R}$ be the solution to
\begin{equation}
   \begin{cases}
      \cfrac{\partial^2 u}{\partial t^2}-\Delta u=f,&\mbox{in}~ \Omega\times\left]0,T     \right],\\
      u=0,&\mbox{on}~ \partial\Omega\times\left]0,T \right],\\
      u(\cdot,0)=u_0,&\mbox{in}~\Omega,\\
      \cfrac{\partial u}{\partial t}(\cdot,0)=v_0,&\mbox{in}~\Omega,
   \end{cases}
   \label{wave}
\end{equation}
where $f,u_0,v_0$ are given functions. Note that if we introduce the auxiliary unknown $v=\frac{\partial u}{\partial t}$ then model (\ref{wave}) can be rewritten as the following first-order in time system 
\begin{equation}
   \begin{cases}
      \cfrac{\partial u}{\partial t}-v=0,  &\mbox{in}~ \Omega\times\left]0,T \right], \\
      \cfrac{\partial v}{\partial t}-\Delta u=f,  &\mbox{in}~ \Omega\times\left]0,T \right], \\
      u=v=0,&\mbox{on}~ \partial\Omega\times\left]0,T \right]  ,\\
      u(\cdot,0)=u_0,~v(\cdot,0)=v_0,~&\mbox{in}~\Omega.
      \label{syst}
   \end{cases}
\end{equation}
The above problem (\ref{wave}) has the following weak formulation \cite{evans2010partial}: for given  \newline 
$f\in L^{2}(0,T;L^2(\Omega))$, $u_0\in H^1_0(\Omega)$ and $v_0\in L^2(\Omega)$ find a function 
\begin{equation}
   u\in L^{2}\left(0,T;H^1_0(\Omega)\right),~\cfrac{\partial u}{\partial t}\in L^{2}\left(0,T;L^2(\Omega)\right),~\cfrac{\partial^2 u}{\partial t^2}\in L^{2}\left(0,T;H^{-1}(\Omega)\right)
\end{equation}
such that $u(x,0)=u_0$ in $H^1_0(\Omega)$, $\cfrac{\partial u}{\partial t}(x,0)=v_0$ in $L^2(\Omega)$ and
\begin{equation}
   \left\langle\cfrac{\partial^2 u}{\partial t^2},\varphi\right\rangle+\left(\nabla u,\nabla \varphi\right)=\left(f,\varphi\right),~\forall\varphi  \in H^1_0(\Omega),
   \label{weakwave}
\end{equation}
where $\left\langle \cdot, \cdot \right\rangle$ denotes the duality pairing between $ H^{-1}(\Omega)$ and $ H^1_0(\Omega)$ and the parentheses $( \cdot, \cdot)$ stand for the inner product in
$L^2 ( \Omega)$. 
Following Chap. 7, Sect. 2, Theorem 5 of \cite{evans2010partial}, we observe that in fact 
\begin{equation*}
   u\in C^{0}\left(0,T;H^1_0(\Omega)\right),~\cfrac{\partial u}{\partial t}\in C^{0}\left(0,T;L^2(\Omega)\right),~\cfrac{\partial^2 u}{\partial t^2}\in C^{0}\left(0,T;H^{-1}(\Omega)\right).
\end{equation*}
Higher regularity results with more regular data are also available in \cite{evans2010partial}.

Let us now discretize (\ref{wave}) or, equivalently, (\ref{syst}) in space using the finite element method and in time using an appropriate marching scheme. We thus introduce a regular mesh $\mathcal{T}_h$ on $\Omega$ with triangles $K$, $\mathrm{diam}~K=h_{K}$, $h=\max_{K\in\Th}h_K$, internal edges $E\in\mathcal{E}_h$, where $\mathcal{E}_h$ represents the internal edges of the mesh $\mathcal{T}_h$  and the standard finite element space ${V}_h\subset H^1_0(\Omega ) $:
$$
V_h=\left\{v_h\in C(\bar\Omega):v_h|_K\in \mathbb{P}_1~\forall K\in\mathcal{T}_h \text{ and }v_h|_{\partial\Omega}=0\right\}.
$$
Let us also introduce a subdivision of the time interval $[0,T]$
$$0=t_0<t_1<\dots<t_N=T$$
with non-uniform time steps $\tau_n=t_{n+1}-t_n$ for $n=0,\ldots,N-1$ and $\tau=\displaystyle\max_{0 \leq n \leq N-1}\tau_n$ . 

The Newmark scheme \cite{newmark1959, RT} with coefficients $\beta =
1/4, \gamma = 1/2$ as applied to the wave equation
(\ref{wave}):
given approximations $u^0_h, v^0_h \in V_h$ of $u_0, v_0$ compute $u^1_h \in V_h$ from
\begin{multline}
   \label{Newm1}\left( \frac{u^1_h - u^0_h}{\tau_0}, \varphi_h \right) + \left( \nabla\frac{\tau_0 (u^1_h + u^0_h)}{4}, \nabla \varphi_h \right) 
    = \left( v_h^0 +\frac{\tau_0}{4} (f^1 + f^0), \varphi_h \right), \\ \forall\varphi_h \in V_h
\end{multline}
and then compute $u^{n+1}_h \in V_h$ for $n = 1, \ldots, N-1$ from equation
\begin{align}
   \left( \frac{u_h^{n + 1} - u_h^n}{\tau_n} - \frac{u_h^n - u_h^{n -1}}{\tau_{n - 1}}, \varphi_h \right) + \left( \nabla \frac{\tau_n (u_h^{n + 1}+ u_h^n) + \tau_{n - 1} (u_h^n + u_h^{n - 1})}{4}, \nabla \varphi_h \right)&\notag\\
   =\left( \frac{\tau_n (f^{n + 1} + f^n) + \tau_{n - 1} (f^n + f^{n - 1})}{4},\varphi_h \right), \hspace{1em} \forall \varphi_h \in V_h.&
   \label{Newm2}
\end{align}
where $f^n$ is an abbreviation for $f ( \cdot,t_n)$.

Following \cite{Baker76} and \cite{our1}, we observe that this scheme is equivalent to the Crank-Nicolson discretization of the governing equation written in the form (\ref{syst}): taking $u^0_h,v^0_h\in V_h$ as some approximations to $u_0,v_0$ compute $u^n_h,v^n_h\in V_h$ for $n=0,\ldots,N-1$ from the system
\begin{align}
\label{CNh1}
\frac{u_h^{n + 1} - u_h^n}{\tau_n} - \frac{v_h^n + v_h^{n+1}}{2} &= 0,
\\
\label{CNh2}
\left( \frac{v^{n + 1}_h - v^n_h}{\tau_n}, \varphi_h \right) + \left( \nabla \frac{u^{n + 1}_h + u^n_h}{2}, \nabla \varphi_h \right) &= \left(  \frac{f^{n+1} + f^n}{2}, \varphi_h \right), \hspace{1em} \forall \varphi_h \in V_h.
\end{align}
Note that the additional unknowns $v_n^h$ are the approximations are not present in the Newmark scheme (\ref{Newm1})--(\ref{Newm2}). If needed, they can be recovered on each time step by the following easy computation
\begin{equation}\label{vhform}
     v_h^{n + 1} = 2 \frac{u_h^{n + 1} - u_h^n}{\tau_n} - v_h^n.
\end{equation}

From now on, we shall use the following notations
\begin{align}\label{notation}
u_h^{n+1/2}& :=\frac{u_h^{n+1}+u_h^{n}}{2},
\quad
\partial _{n+1/2}u_h:=\frac{u_h^{n+1}-u_h^{n}}{\tau_n},
\quad
\partial _{n}u_h:=\frac{u_h^{n+1}-u_h^{n-1}}{\tau_n+\tau_{n-1}} \\
\notag \partial _{n}^{2}{u_h}&:=\frac{1}{\tau_{n-1/2}}\left(\frac{u_h^{n+1}-u_h^{n}}{\tau_n}-\frac{u_h^{n}-u_h^{n-1}}{\tau_{n-1}}\right)
\text{ with } \tau_{n-1/2}:=\frac{\tau_n+\tau_{n-1}}{2}
\end{align}
We apply this notations to all quantities indexed by a superscript, so that, for example, $f^{n+1/2}=({f^{n+1}+f^n})/{2}$.
We also denote $u (x,t_n)$, $v(x,t_n)$ by $u^n$, $v^n$ so that, for example, $u^{n+1/2}=\left({u^{n+1}+u^n}\right)/{2}=\left(u(x,t_{n+1})+u(x,t_n)\right)/{2}$.

We shall measure the error in the following norm
\begin{equation}\label{EnNorm}
u \mapsto \max_{t\in[0,T]}\left(\left\|\cfrac{\partial u}{\partial t} (t)\right\|_{L^2(\Omega)}^{2}+\left\vert u(t)\right\vert^{2}_{H^1(\Omega)}\right) ^{1/2}.
\end{equation}
Here and in what follows, we use  the notations $u(t)$ and $\cfrac{\partial u}{\partial t} (t)$ as a shorthand for, respectively, $u(\cdot,t)$ and $\cfrac{\partial u}{\partial t} (\cdot,t)$. The norms and semi-norms in Sobolev spaces $H^k(\Omega)$ are denoted, respectively, by $\|\cdot\|_{H^k(\Omega)}$ and $|\cdot|_{H^k(\Omega)}$.
We call (\ref{EnNorm}) the energy norm referring to the underlying physics of the studied phenomenon. Indeed, the first term in (\ref{EnNorm}) may be assimilated to the kinetic energy and the second one to the potential energy.

Let us introduce $e^n_u = u^n_h - \Pi_h u^n$ and $e^n_v = v^n_h - I_h v^n$ where
$\Pi_h : H^1_0 (\Omega) \to V_h$ is the $H^1_0$-orthogonal projection
operator, i.e.
\begin{equation}\label{Pih}
 \left(\nabla \Pi_h v, \nabla\varphi_h\right) = \left(\nabla v,
   \nabla\varphi_h\right), \hspace{1em} \forall v \in H^1_0 (\Omega),\hspace{1em}\forall \varphi_h \in V_h 
\end{equation}
and $\tilde I_h : H^1_0 (\Omega) \to V_h$ is a Cl{\'e}ment-type interpolation operator which is also a projection, i.e. $\tilde I_hV_h=V_h$ \cite{ErnGue, ScoZh}. 

Let us recall, for future reference, the well known properties of these operators (see  \cite{ErnGue}): for every sufficiently smooth function $v$ the following inequalities hold
\begin{equation}\label{Pinterp}
  | \Pi_h v |_{H^1(\Omega)}\leq | v |_{H^1(\Omega)}, 
  \hspace{1em} | v - \Pi_h v |_{H^1(\Omega)} \leq Ch | v |_{H^2(\Omega)}
\end{equation}
with a constant $C > 0$ which depends only on the regularity of the mesh.
Moreover, for all $K\in\Th$ and  $E \in \mathcal{E}_h$ we have 
\begin{equation}\label{Clement}
 \| v - I_h v \|_{L^2(K)} \leq Ch_K |v |_{H^1(\omega_K)}
\text{ and }
\| v - I_h v \|_{L^2(E)} \leq Ch_{E}^{1 / 2} | v |_{H^1(\omega_{E})}
\end{equation}
Here $\omega_K$  (resp. $\omega_E$) represents the set of triangles of $\mathcal{T}_h$ having a common vertex with triangle $K$ (resp. edge $E$) and the constant $C > 0$ depends only on the regularity of the mesh.
\section{The 3-point time error estimator}\label{sectionRecall} 

The aim of this section is to recall \textit{a posteriori} bounds in time and space from \cite{our1} for the error measured in the norm (\ref{EnNorm}). Their derivation is based on the following piecewise quadratic (in time) 3-point reconstruction of the discrete solution.  
\begin{definition}\label{QuadRec}
Let $u^n_h$ be the discrete solution given by the scheme (\ref{Newm2}).  Then, the piecewise quadratic reconstruction $\tilde{u}_{h\tau} (t) : [0, T] \rightarrow V_h$ is constructed as the continuous in time function that is equal on $[t_n, t_{n + 1}]$, $n\ge 1$, to the quadratic polynomial in $t$ that coincides with $u^{n + 1}_h$ (respectively $u^n_h$, $u^{n - 1}_h$) at time $t_{n+1}$ (respectively $t_n$, $t_{n - 1}$).  Moreover, $\tilde{u}_{h\tau} (t)$ is defined on $[t_0, t_{1}]$ as the quadratic polynomial in $t$ that coincides with $u^{2}_h$ (respectively $u^1_h$, $u^{0}_h$) at time $t_{2}$ (respectively $t_1$, $t_{0}$). Similarly, we introduce  piecewise quadratic reconstruction $\tilde{v}_{h\tau} (t) : [0, T] \rightarrow V_h$ based on $v^n_h$ defined by (\ref{vhform}) and $\tilde{f}_{\tau} (t) : [0, T] \rightarrow L^2(\Omega)$ based on $f(t_n,\cdot)$.  
\end{definition}
The quadratic reconstructions $\tilde{u}_{h\tau}$, $\tilde{v}_{h\tau}$ are thus based on three points in time (normally looking backwards in time, with the exemption of the initial time slab $[t_0,t_1]$). This is also the case for the time error estimator (\ref{in}), recalled in the following Theorem and therefore referred to as the $3$-point estimator.

\begin{theorem}\label{lemest3}
The following \textit{a posteriori} error estimate holds between the solution $u$ of the wave equation (\ref{wave}) and the discrete solution $u_h^n$ given by (\ref{Newm1})--(\ref{Newm2}) for all $ t_n,~0\leq n\leq N$ with $v_h^n$ given by (\ref{vhform}):
\begin{multline}
\left(\left\Vert v^{n}_h- \cfrac{\partial u}{\partial t} (t_n)\right\Vert_{L^2(\Omega)} ^{2}+\left\vert u^{n}_h-u(t_{n})\right\vert ^{2}_{H^1(\Omega)}\right) ^{1/2}\\
\leq\left(\left\Vert v^{0}_h-v_0\right\Vert_{L^2(\Omega)} ^{2}+\left\vert u^{0}_h-u_0\right\vert ^{2}_{H^1(\Omega)}\right) ^{1/2} \\
+\eta _{S}(t_{n})+\sum_{k=1}^{n}\tau_{k-1}\eta _{T}(t_{k-1})
+\int_0^{t_n} \|f-\tilde{f}_\tau\|_{L^2(\Omega)}dt
\label{estf}
\end{multline}
where the space indicator is defined by
\begin{align}
  \eta_S (t_n)
  &= C_1 \max_{0 \leqslant t \leqslant t_n} \Biggl[ \sum_{K \in \mathcal{T}_h}
   h_K^2  \left\Vert \frac{\partial \tilde{v}_{h \tau}}{\partial t} - \Delta \tilde{u}_{h \tau}-f
   \right\Vert_{L^2(K)}^2\label{space}\\
   \notag&~~~~~~~~~~~~~~~~~~~~~~~~~~~~~~~~~~~~~~~~~~~~~+ \sum_{
E \in \mathcal{E}_h}h_{E} \left|\left[n \cdot \nabla \tilde{u}_{h
   \tau}\right]\right|_{L^2(E)}^2 \Biggl]^{1/2}
\\
\notag &+ C_2\sum_{m = 0}^{n-1} \int_{t_m}^{t_{m + 1}} \Biggl[ \sum_{K \in \mathcal{T}_h} h_K^2
  \left\Vert \frac{\partial^2 \tilde{v}_{h \tau}}{\partial t^2} - \Delta \frac{\partial
  \tilde{u}_{h\tau}}{\partial t} -\frac{\partial{f}}{\partial{t}}\right\Vert_{L^2(K)}^2\\
  \notag &~~~~~~~~~~~~~~~~~~~~~~~~~~~~~~~~~~~~~~~~+ \sum_{
E \in \mathcal{E}_h}h_{E} \left\Vert\left[n \cdot
  \nabla \frac{\partial \tilde{u}_{h\tau}}{\partial t}\right]\right\Vert_{L^2(E)}^2
  \Biggr]^{1/2}dt,
\end{align}
here  $C_1,~C_2$ are constants depending only on the mesh regularity,
$[\cdot]$ stands for a jump on an edge $E\in\mathcal{E}_h$, and 
$\tilde{u}_{h\tau}$, $\tilde{v}_{h\tau}$ are given by Definition \ref{QuadRec}. 

The error indicator in time for $k=1,\dots,N-1$ is
\begin{equation}\label{in}
\eta_T (t_k)=\left(\frac{1}{12}\tau_{k}^2+\frac{1}{8}\tau_{k-1}\tau_{k}\right)\left(\left\vert\partial _{k}^{2}{v_h}\right\vert_{H^1(\Omega)}
+ \left\Vert \partial _{k}^{2}{f_h} - z_h^k\right\Vert_{L^2(\Omega)}^2\right)^{1/2}
\end{equation}
where $z^k_h$ is such that
\begin{equation}\label{zh}
\left(z_h^k, \varphi_h\right) =(\nabla \partial _{k}^{2}{u_h}, \nabla\varphi_h),
\quad\forall\varphi_h\in V_h
\end{equation}
and
\begin{equation}\label{in0step}
\eta_T (t_0)=\left(\frac{5}{12}\tau_{0}^2+\frac{1}{2}\tau_{1}\tau_{0}\right)\left(\left\vert\partial _{1}^{2}{v_h}\right\vert_{H^1(\Omega)}
+ \left\Vert \partial _{1}^{2}{f_h} - z^1_h\right\Vert_{L^2(\Omega)}^2\right)^{1/2}
\end{equation}
\end{theorem}

We also recall the optimality result for the $3$-point time error estimator.
\begin{theorem}\label{optimality}
Let u be the solution of wave equation (\ref{wave}) and 
$\displaystyle\cfrac{\partial^{3}u}{\partial t^{3}}(0)\in {H^1(\Omega)}$, $\displaystyle\cfrac{\partial^{2}u}{\partial t^{2}}(0) \in {H^2(\Omega)}$, $\displaystyle\cfrac{\partial^{2}f}{\partial t^{2}}(t)\in L^{\infty}(0,T;{L^2(\Omega)})$, $\displaystyle\cfrac{\partial^{3}f}{\partial t^{3}}(t)\in L^{2}(0,T;{L^2(\Omega)})$. Suppose that mesh $\Th$ is quasi-uniform, the mesh in time is uniform ($t_k=k\tau$), and the initial approximations are chosen as
\begin{equation}\label{initPi}
u_h^0=\Pi_hu_0, \quad v_h^0=\Pi_hv_0 .
\end{equation}
Then, the 3-point time error estimator $\eta_T(t_k)$ defined by (\ref{in}, \ref{in0step}) is of order $\tau^2$, i.e. 
\begin{equation*}
\eta_T(t_k)\leq C \tau^2.
\end{equation*}
with a positive constant $C$ depending only on $u$, $f$, and the regularity of mesh $\Th$.
\end{theorem}
\begin{remark}\label{intImp}
Note that the approximation of initial conditions (\ref{initPi}) is crucial for optimality of our time and space error estimators. As explained theoretically and demonstrated numerically in \cite{our1}, this is due to the boundedness of higher order discrete derivatives that contains in time and space error estimators (\ref{in}, \ref{in0step}, \ref{space}). 
\end{remark}

\section{The 5-point \textit{a posteriori} error estimator}\label{section3}

As already mentioned in the Introduction, the time error estimator (\ref{in}) contains a finite element approximation to the Laplacian of $u_h^k$, i.e. $z_h^k$ given by (\ref{zh}). This is unfortunate because $z_h^k$ should be computed by solving an additional finite element problem that implies additional computational
effort. Having in mind that the term $\partial_n^2 f_h - z_h^n$ in (\ref{in}) is
a discretization of ${\partial^2 f}/{\partial t^2} + \Delta u = {\partial^4 u}/{\partial t^4}$ at time $t_n$  our goal now is to avoid the second derivatives in space in the error estimates and replace them with the forth derivatives in time. 

We introduce a ``fourth order finite difference in time''  $\partial_n^4$ for any $w_h^n\in V_h$, $k=0,1,\ldots$ as
\begin{equation}\label{newd4}
\partial_n^4 w_h = \hat{\partial}^2_n \partial^2_{} w_h 
\end{equation}
where 
\begin{align}\label{newd2}
\hat{\partial}^2_n w_h &= \frac{2}{(\hat{t}_n - \hat{t}_{n - 2})}\left(\frac{w_h^n - w_h^{n - 1}}{\hat{t}_n - \hat{t}_{n
   - 1}} - \frac{w_h^{n - 1} - w_h^{n - 2}}{\hat{t}_{n - 1} - \hat{t}_{n -
   2}}\right),\\[1mm]
\notag\hat{t}_n &= \frac{t_{n + 1} + t_{n - 1}}{2}
\end{align}
and $\partial^2 w_h$ stands for the collection of $\partial^2_k w_h$ with $k=1,2,\ldots$, so that the notation $\hat{\partial}^2_n \partial^2 w_h$ means that one should substitute $\partial^2_k w_h$ for $w^k_h$, $k=n,n-1,n-2$, in the definition of $\hat{\partial}^2_n w_h$.
In particular,
$\partial_n^4 w_h$ is indeed the standard fourth order finite difference in
time at $t_{n-1}$ of $w_h$ in the case of constant time steps $\tau_n = \tau$:
\[ \partial^4_n w_h = \frac{w_h^{n + 1} - 4 w_h^n + 6 w_h^{n - 1} - 4 w_h^{n -
   2} + w_h^{n - 3}}{\tau^4} \]
However, $\partial^4_n$ is not necessarily consistent with the fourth time
derivative in the general case $\tau_n \neq \text{const}$. 

We denote 
\begin{equation}\label{bar}
\bar{w}_h^n = \frac{\tau_n (w_h^{n + 1}+ w_h^n) + \tau_{n - 1} (w_h^n + w_h^{n - 1})}{4\tau_{n-1/2}}.
\end{equation}
Consistently with the conventions above, $\bar{w}_h$ will stand for the collection of $\bar{w}_h^k$, $k=0,1,\ldots$

We shall need a connection between second order discrete derivatives $\hat{\partial}^2_n$ and $\partial^2_n$  which we establish in the following technical lemma.
\begin{lemma}\label{lemtech1}
For all integer  $n=3,\ldots N-1$ there exist coefficients $\alpha_k$, $k=n-2,n-1,n$ such that for all $\{w_h^n\}$
\begin{equation}
 \label{sumalpha} \hat{\partial}^2_n  \bar{w}_h = \sum_{k = n - 2}^n \alpha_k \partial_k^2
   w_h
\end{equation}
Moreover 
$$|\alpha_k|\leq c, \mbox{ for }k=n-2,n-1,n,  \mbox{ and }\displaystyle\sum_{k = n - 2}^n \alpha_k \ge C$$ 
where $c$ and $C$ are positive constants depending only on the mesh regularity in time, i.e. on $\max_{k\ge 0}\left( \frac{\tau_{k+1}}{\tau_k} + \frac{\tau_k}{\tau_{k+1}} \right)$.
\end{lemma}
\begin{proof}
We first note that relation (\ref{sumalpha}) does not contain any derivatives in space and thus it should hold at any point $x\in\Omega$. Consequently, it is sufficient to prove this Lemma assuming that $w^n_h$, $\partial^2_n w_h$, etc. are real numbers, i.e. replacing $V_h$ by $\mathbb{R}$. This is the assumption adopted in this proof. We shall thus drop the sub-indexes $h$ everywhere. Furthermore, it will be convenient to reinterpret $w^n$ in (\ref{newd4}), (\ref{newd2}) and (\ref{bar})  as the values of a real valued function $w(t)$ at $t=t_n$. We shall also use the notations like $\bar{w}^n$, $\partial_n^2 w$, and so on, where $w$ is a continuous function on $\mathbb{R}$, always assuming $w^n=w(t_n)$. 

Observe that $\hat{\partial}^2_n \bar{w}$ is a linear combination of 5 numbers $\{w^{n - 3}, \ldots, w^{n + 1} \}$. Thus, it is enough to check equality (\ref{sumalpha}) on any 5 continuous functions $\phi_{(k)}(t)$,  $k=n-3,\ldots,n+1$, such that the vector of values of $\phi_{(k)}$ at times $t_l$, $l=n-3,\ldots,n+1$, form a basis of $\mathbb{R}^5$. For fixed $n$, let us choose these functions as
\begin{align}
\label{LinFunct} \phi_{(k)}(t)&=\begin{cases}\displaystyle\cfrac{t-t_{k-1}}{\tau_{k-1}}, & \mbox{if}~ t<t_{k},\\
\displaystyle\cfrac{t_{k+1}-t}{\tau_{k}}, & \mbox{if}~ t\geq t_{k},
\end{cases}  
\qquad k=n-3,\ldots,n+1.
\end{align}

First we notice that for every linear function $u(t)$ on $\left[t_{n-3}, t_{n+1}\right]$ we have $\hat{\partial}^2_n 
\bar{u} = \partial_n^2 u = 0$. Thus, we get immediately $\hat{\partial}^2_n \bar{\phi}_{(n-3)} = \partial_n^2 \phi_{(n-3)} = 0$ and $\hat{\partial}^2_n \bar{\phi}_{(n+1)} = \partial_n^2 \phi_{(n+1)} = 0$ so that (\ref{sumalpha}) is fulfilled on functions $\phi_{(n-3)}$, $\phi_{(n+1)}$ with any coefficients $\alpha_k$, $k=n-2,n-1,n$. Now we want to provide coefficients $\alpha_k$, $k=n-2,n-1,n$ for which (\ref{sumalpha}) is fulfilled  on functions ${\phi}_{(n-2)}$, ${\phi}_{(n-1)}$ and ${\phi}_{(n)}$. For brevity, we demonstrate the idea only for function $\phi_{(n)}(t)$. Function $\phi_{(n)}(t)$ is linear on $\left[t_{n-3}, t_{n}\right]$ and thus
$$\partial_{n-2}^2 \phi_{(n)} = 0,~\partial_{n-1}^2 \phi_{(n)} = 0.$$
From direct computations it is easy to show that 
$$\partial_{n}^2 \phi_{(n)} \sim \frac{1}{\tau_n^2},~\bar{\phi}_{(n)} \sim 1,~\hat{\partial}^2_n \bar{\phi}_{(n)}\sim \frac{1}{\tau_n^2}$$ 
where $\sim$ hides some factors that can be bounded by constants depending only on the mesh regularity. Thus we are able to establish expression for coefficient $\alpha_{n}=\displaystyle\cfrac{\hat{\partial}^2_n \bar{\phi}_{(n)}}{\partial_{n}^2 \phi_{(n)}}\leq C$. Similar reasoning for function $\phi^{n-1}_n$ and $\phi^{n-2}_n$ shows that $\alpha_{n-1}=\displaystyle\cfrac{\hat{\partial}^2_n \bar{\phi}_{(n-1)}}{\partial_{n-1}^2 \phi_{(n-1)}}\leq C$ and  $\alpha_{n-2}=\displaystyle\cfrac{\hat{\partial}^2_n \bar{\phi}_{(n-2)}}{\partial_{n-2}^2 \phi_{(n-2)}}\leq C$.

The next step is to show boundedness from below of $\displaystyle\sum_{k = n - 2}^n \alpha_k$. We will show it by applying equality (\ref{sumalpha}) to second order polynomial function $s(t)=\displaystyle\frac{t^2}{2}$. 
Using a Taylor expansion of $s(t)$ around $\hat{t}_n$ in the definition of $\bar{s}^n$ gives
\begin{align*}
\notag \bar{s}^n&=\frac{\tau_n(\hat{t}_n^2+\hat{t}_n\tau_{n-1}+\displaystyle\frac{1}{4}(\tau_n^2+\tau_{n-1}^2)) + \tau_{n-1}(\hat{t}_n^2-\hat{t}_n\tau_{n}+\displaystyle\frac{1}{4}(\tau_n^2+\tau_{n-1}^2))}{2(\tau_n+\tau_{n-1})}
\\
&=\frac{\hat{t}_n^2}{2}+{\frac{1}{8}\left(\tau_n^2+\tau_{n-1}^2\right)}
\label{mid}
\end{align*}
Substituting this into the definition of $\hat{\partial}_n^2\bar s$  we obtain 
\begin{align*}
 \hat{\partial}_n^2\bar{s}= \displaystyle\frac{\displaystyle\frac{\bar{s}^n - \bar{s}^{n - 1}}{\hat{t}_n - \hat{t}_{n- 1}} - \frac{\bar{s}^{n - 1} - \bar{s}^{n - 2}}{\hat{t}_{n - 1} - \hat{t}_{n -
  2}}}{\left(\hat{t}_n - \hat{t}_{n - 2}\right) / 2}& = 1+{\frac{1}{8}}\left(\displaystyle\frac{2\displaystyle\frac{\tau_n^2 - \tau^2_{n-2}}{\tau_n + \tau_{n-2}} - 2\displaystyle\frac{\tau_{n-1}^2 - \tau^2_{n-3}}{\tau_{n-1} + \tau_{n-3}}}{\frac{1}{4}(\tau_n + \tau_{n-1} + \tau_{n-2} + \tau_{n-3})}\right)\\ 
\notag&=1+\displaystyle\cfrac{\tau_{n}-\tau_{n-1}-\tau_{n-2}+\tau_{n-3}}{\tau_{n}+\tau_{n-1}+\tau_{n-2}+\tau_{n-3}}
\end{align*}
Using (\ref{sumalpha}) and the fact that $\partial^2_ks=1$ for $k=n-2,n-1,n$ we note that 
\begin{equation*}
1+\displaystyle\cfrac{\tau_{n}-\tau_{n-1}-\tau_{n-2}+\tau_{n-3}}{\tau_{n}+\tau_{n-1}+\tau_{n-2}+\tau_{n-3}}=\sum_{k = n - 2}^n \alpha_k 
\end{equation*}
This implies $\displaystyle\sum_{k = n - 2}^n \alpha_k\geq C$. 
\end{proof}

\begin{lemma}\label{lemtech2}
Let $w_h^n, s_h^n\in V_h$ be such that  
\begin{equation}
\label{Nick} \frac{w_h^{n + 1} - w_h^n}{\tau_n} - \frac{s_h^n + s_h^{n +1}}{2} = 0,~ \forall n\geq0.
\end{equation}
For all $n\geq 3$ there exist coefficients $\beta_k$, $k=n-2,n-1,n$ such that
\begin{equation}
 \label{sumbeta} \sum_{k = n - 2}^n \alpha_k \partial_k^2
   w_h = \left(\sum_{k = n - 2}^n \alpha_k\right)\partial_n^2 w_h - \tau_n  \sum_{k = n - 2}^n \beta_k \partial_k^2s_h
\end{equation}
where coefficients $\alpha_k$, $k=n-2,n-1,n$ are introduced in Lemma \ref{lemtech1}. 
Moreover
$$|\beta_k|\leq C,~k=n-2,n-1,n$$
where $C$ is a positive constant depending only on the mesh regularity in time, i.e. on $\max_{k\ge 0}\left( \frac{\tau_{k+1}}{\tau_k} + \frac{\tau_k}{\tau_{k+1}} \right)$.
\end{lemma}
\begin{proof}
As in proof of Lemma \ref{lemtech1}, we assume $V_h=\mathbb{R}$, drop the sub-indexes $h$ and interpret $w^n, s^n$ as the values of continuous real valued functions $w(t)$, $s(t)$ at $t=t_n$. Using (\ref{Nick}) and notations (\ref{notation}) implies $\partial_k^2w=\partial_ks$. Now, we are able to rewrite (\ref{sumbeta}) in terms of $s^n$ only 
\begin{equation}
 \label{sumbetaS} \sum_{k = n - 2}^n \alpha_k {\partial_ks}= \left(\sum_{k = n - 2}^n \alpha_k\right) {\partial_ns} - \tau_n  \sum_{k = n - 2}^n \beta_k \partial_k^2s
\end{equation}
As in the proof of Lemma \ref{lemtech1} we take into account the fact that equation (\ref{sumbetaS}) should hold for every 5 numbers $\{s^{n - 3}, \ldots, s^{n + 1} \}$  and therefore it's enough to check equality (\ref{sumbetaS}) on 5 linearly independent piecewise linear functions $\phi_{(k)}$ introduced by (\ref{LinFunct}). Using the reasoning as in Lemma \ref{lemtech1} leads to desired result (\ref{sumbeta}).
\end{proof}

We can now prove an \textit{a posteriori} error estimate involving
$\partial^{4}_{n}u_{h}$. Since the latter is computed
through 5 points in time $\{t_{n - 3}, \ldots, t_{n + 1} \}$, we shall refer
to this approach as the 5-point estimator. For the same reason, this estimator
is only applicable from time $t_4$. The error at first 3 time steps should be
thus measured differently, for example using the 3-point estimator from Theorem
\ref{lemest3}. The deriving of 5-point time error estimator is based on main ideas from the proof of Theorem \ref{lemest3} using Lemma \ref{lemtech1}. 

\begin{theorem}\label{lemest5}
The following \textit{a posteriori} error estimate holds between the solution $u$ of the wave equation (\ref{wave}) and the discrete solution $u_h^n$ given by (\ref{Newm1})--(\ref{Newm2}) for all $ t_n,~4\leq n\leq N$ with $v_h^n$ given by (\ref{vhform}):
\begin{multline}\label{estf5}
\left(\left\Vert v^{n}_h- \cfrac{\partial u}{\partial t} (t_n)\right\Vert_{L^2(\Omega)} ^{2}+\left\vert u^{n}_h-u(t_{n})\right\vert ^{2}_{H^1(\Omega)}\right) ^{1/2}\\
\leq\left(\left\Vert v^{3}_h- \cfrac{\partial u}{\partial t} (t_3)\right\Vert_{L^2(\Omega)} ^{2}+\left\vert u^{3}_h-u(t_{3})\right\vert ^{2}_{H^1(\Omega)}\right) ^{1/2}\\
+\eta _{S}(t_{N})+\sum_{k=3}^{N-1}\tau_k\hat\eta _{T}(t_{k})
+h.o.t.
\end{multline}
where the space indicator is defined by (\ref{space}) and
the error indicator in time is
  \begin{equation}
    \hat\eta_T (t_k) = \left(\frac{1}{12}\tau_{k}^2+\frac{1}{8}\tau_{k-1}\tau_{k}\right) \left(\left| \partial_k^2 v_h \right|_{H^1(\Omega)} +\left\|
    \partial_k^4 u_h \right\|_{L^2(\Omega)}\right)  \label{in5}
  \end{equation}
with a constant $C>0$ depending only on the mesh regularity in time, i.e. on $\max_{k\ge 0}\left( \frac{\tau_{k+1}}{\tau_k} + \frac{\tau_k}{\tau_{k+1}} \right)$. Finally. h.o.t. (higher order terms) in (\ref{estf5}) hide terms of order $O(\tau^3)$.  
\end{theorem}

\begin{proof}
Rewrite scheme (\ref{Newm2})  as
\begin{equation}
  \partial_n^2 u_h + A_h \bar{u}_h^n = \bar{f}_h^n
  \label{Newmh}
\end{equation}
for $n = 0, \ldots, N-1$ where $f_h^n=P_h f(t_n,\cdot)$ and operator $A_h$ defined in (\ref{Ah}). 
Taking a linear combination of instances of (\ref{Newmh}) at steps $n, n - 1,
n - 2$ with appropriate coefficients gives
\begin{equation}
  \label{Newmhd} \partial_n^4 u_h + A_h  \hat{\partial}^2_n  \bar{u}_h =
  \hat{\partial}^2_n  \bar{f}_h
\end{equation}
Using the definition of operator $\hat{\partial}_n^2$ and re-introducing
$v_h^n$ by (\ref{CNh1}) leads to
\[  \hat{\partial}^2_n  \bar{u}_h = \sum_{k = n - 2}^n \alpha_k \partial_k^2
   u_h = \left(\sum_{k = n - 2}^n \alpha_k\right)\partial_n^2 u_h - \tau_n  \sum_{k = n - 2}^n \beta_k \partial_k^2
   v_h \]
withcoefficients $\alpha_k, \beta_k$ introduced in Lemma \ref{lemtech1} and Lemma \ref{lemtech2}. Moreover, by Lemma \ref{lemtech1} $\gamma=\left(\sum_{k = n - 2}^n \alpha_k\right)^{-1}$ is positive and bounded so that  
\[  
  \partial_n^2 u_h = \gamma \hat{\partial}^2_n   \bar{u}_h
     + \tau_n  \sum_{k = n - 2}^n \gamma_k \partial_k^2   v_h 
\]
with $\gamma_k = \gamma\beta_k$
that are all uniformly bounded on regular meshes in time.
Similarly,
\[ 
  \partial_n^2 f_h
   = \hat{\partial}^2_n  \bar{f}_h + \tau_n  \sum_{k = n - 2}^n \gamma_k \partial_k^2 
   \dot{f}_h \]
with $\dot{f}_h^n$ satisfying
\[ \frac{f_h^{n + 1} - f_h^n}{\tau_n} = \frac{\dot{f}_h^n +
   \dot{f}_h^{n + 1}}{2} \]
Thus,
\begin{align} 
\notag \partial^2_n f_h - A_h \partial^2_n u_h & = \hat{\partial}^2_n  \bar{f}_h -
   A_h  \hat{\partial}^2_n  \bar{u}_h + \tau_n  \sum_{k = n - 2}^n \gamma_k
   \left(\partial_k^2  \dot{f}_h - A_h \partial_k^2 v_h\right) \\
\notag  & = \partial_n^4 u_h +
   \tau_n  \sum_{k = n - 2}^n \gamma_k \left(\partial_k^2  \dot{f}_h - A_h
   \partial_k^2 v_h\right) 
   \end{align}
We can now reproduce the proof of Theorem \ref{lemest3}. In the following, we adopt the vector notation $U (t, x)
=\begin{pmatrix}
  u (t, x)\\
  v (t, x)
\end{pmatrix}$ where $v = {\partial u}/{\partial t}$. Note that the first equation in
(\ref{syst}) implies that
\[ \left(\nabla \cfrac{\partial u}{\partial t}, \nabla \varphi\right) - (\nabla v, \nabla \varphi) = 0,
   \quad \forall \varphi \in H^1_0 (\Omega) \]
by taking its gradient, multiplying it by $\nabla \varphi$ and integrating
over $\Omega$. Thus, system (\ref{syst}) can be rewritten in the vector
notations as
\begin{equation}
  {b} \left(\cfrac{\partial U}{\partial t}, \Phi\right) + \left(\mathcal{A} \nabla U, \nabla \Phi\right) = 
  {b} (F,\Phi), \quad  \forall \Phi \in (H^1_0 (\Omega))^2 \label{ODEf}
\end{equation}
where $\mathcal{A} = \begin{pmatrix}
  0 &- 1\\
  1&~0
\end{pmatrix}$, $F =\begin{pmatrix}
  0\\
  f
\end{pmatrix}$ and
$$ {b} ( U, \Phi)
 = {b}\left(\begin{pmatrix}u\\ v\end{pmatrix}, \begin{pmatrix} \varphi\\ \psi\end{pmatrix} \right)
:= (\nabla u, \nabla \varphi) + (v, \psi)
$$
Similarly, Newmark scheme (\ref{CNh1})--(\ref{CNh2}) can be rewritten as
\begin{equation}
  {b} \left( \frac{U_h^{n + 1} - U_h^n}{\tau_n}, \Phi_h \right) +
  \left( \mathcal{A} \nabla \frac{U_h^{n + 1} + U_h^n}{2}, \nabla \Phi_h \right) = 
  {b} \left(F^{n+1/2}, \Phi_h\right), \hspace{1em} \forall \Phi_h \in V_h^2
  \label{vectorScheme}
\end{equation}
where $U_h^n = \begin{pmatrix}
  u_h^n\\
  v_h^n
\end{pmatrix}$  and $F^{n+1/2} = \begin{pmatrix}
  0\\
  f^{n + 1/2}
\end{pmatrix}$.

The \textit{a posteriori} analysis relies on an appropriate residual equation for the quad\-ra\-tic reconstruction $\tilde{U}_{h\tau}=\begin{pmatrix} \tilde{u}_{h\tau}\\ \tilde{v}_{h\tau}\end{pmatrix}$. We have thus for $t \in [t_n,t_{n + 1}]$, $n = 1, \ldots, N-1$
\begin{equation}
  \tilde{U}_{h\tau} (t) = U^{n + 1}_h + (t - t_{n + 1}) \partial_{n +
  1/2} U_h + \frac{1}{2}  (t - t_{n + 1})  (t - t_n) \partial_n^2 U_h
\end{equation}
so that, after some simplifications,
\begin{multline}\label{Uode1}
{b} \left( \frac{\partial \tilde{U}_{h\tau}}{\partial t}, \Phi_h
   \right) + (\mathcal{A} \nabla \tilde{U}_{h\tau}, \nabla \Phi_h)
   ={b} \left( (t - t_{n + 1/2}) \partial_n^2 U_h + F^{n+1/2},
   \Phi_h \right) \\
   + \left( (t - t_{n + 1/2}) \mathcal{A} \nabla
   \partial_{n + 1/2} U_h + \frac{1}{2} (t - t_{n + 1}) (t - t_n)
   \mathcal{A} \nabla \partial_n^2 U_h, \nabla \Phi_h \right)
\end{multline}
Consider now (\ref{vectorScheme}) at time steps $n$ and $n - 1$. Subtracting one from another and dividing by $\tau_{n - 1/2}$ yields
$$
{b} \left(\partial_n^2 U_h, \Phi_h\right) + \left(\mathcal{A} \nabla \partial_n
   U_h, \nabla \Phi_h\right) = {b}\left( \partial_n F, \Phi_h \right) 
$$
or
$$
{b} \left(\partial_n^2 U_h, \Phi_h\right) + \left(\mathcal{A} \nabla \left(
   \partial_{n + 1/2} U_h - \frac{\tau_{n - 1}}{2} \partial_n^2 U_h
   \right), \nabla \Phi_h \right) = 
   {b}\left(\partial_n F, \Phi_h \right) 
$$
so that (\ref{Uode1}) simplifies to
\begin{multline}\label{Uode2}
 {b} \left(\frac{\partial \tilde{U}_{h\tau}}{\partial t}, \Phi_h
   \right) + \left(\mathcal{A} \nabla \tilde{U}_{h\tau}, \Phi_h\right) \\
   = \left(p_n \mathcal{A} \nabla \partial_n^2 U_h, \nabla \Phi_h\right) + 
   {b}\left( \left(t - t_{n + 1/2}\right) \partial_n F + F^{n+1/2}, \Phi_h
   \right) \\
   = \left(p_n \mathcal{A} \nabla \partial_n^2 U_h, \nabla \Phi_h\right) + 
   {b}\left( \tilde{F}_\tau - p_n \partial^2_n F, \Phi_h\right) 
\end{multline}
where
\begin{align*} p_n &= \frac{\tau_{n - 1}}{2}  (t - t_{n + 1/2}) + \frac{1}{2}
   (t - t_{n + 1})  (t - t_n), \\
  \tilde{F}_{\tau} (t) &= F^{n + 1}_h + (t - t_{n + 1}) \partial_{n +
  1/2} F + \frac{1}{2}  (t - t_{n + 1})  (t - t_n) \partial_n^2 F.
\end{align*}

Introduce the error between reconstruction $\tilde{U}_{h\tau}$ and solution
$U$ to problem (\ref{ODEf}) :
\begin{equation}
  E = \tilde{U}_{h\tau} - U
\end{equation}
or, component-wise
$$
E = \begin{pmatrix}
     E_u\\
     E_v
   \end{pmatrix} = \begin{pmatrix}
     \tilde{u}_{h\tau} - u\\
     \tilde{v}_{h\tau} - v
   \end{pmatrix}
$$
Taking the difference between (\ref{Uode2}) and (\ref{ODEf}) we obtain the residual differential equation for the error valid for $t \in [t_n, t_{n
+ 1}]$, $n = 1, \ldots, N-1$
\begin{align}  
\label{errode}{b}({\partial}_{t}E,{\Phi})+(\mathcal{A}{\nabla} E,{\nabla} {\Phi})
  &={b} \left(\cfrac{{\partial} \tilde{U}_{{\tau} h}}{\partial t}-F,{\Phi}-{\Phi}_{h}\right)+\left(\mathcal{A}{\nabla} \tilde{U}_{{\tau} h},{\nabla} ({\Phi}-{\Phi}_{h})\right)\\
\notag  +\left(p_n \mathcal{A}{\nabla} {\partial}_{n}^{2} U_{h},{\nabla} {\Phi}_{h}\right)
  &+{b}\left( \tilde{F}_\tau -F - p_n \partial^2_n F, \Phi_h\right), \hspace{1em} \forall \Phi_h \in V_h^2 
\end{align}

Now we take $\Phi = E$, $\Phi_h = \begin{pmatrix}
  \Pi_h E_u\\
  \tilde I_h E_v
\end{pmatrix}$ where $\Pi_h : H^1_0 (\Omega) \to V_h$ is the
$H^1_0$-orthogonal projection operator (\ref{Pih}) and $\tilde I_h : H^1_0 (\Omega) \to V_h$ is a
Cl{\'e}ment-type interpolation operator which is also a projection \cite{ErnGue, ScoZh}. Noting that $( \mathcal{A} \nabla E,
\nabla E) = 0$ and
$$
\left( \nabla \cfrac{\partial \tilde{u}_{h\tau}}{\partial t}, \nabla ( E_u -
\Pi_h E_u)\right) = \left( \nabla \tilde{v}_{h\tau}, \nabla \left( E_u - \Pi_h E_u\right)\right) = 0
$$
Introducing operator $A_h:V_h\to V_h$ such that
\begin{equation}\label{Ah}
\left(A_h w_h, \varphi_h\right) =(\nabla w_h, \nabla\varphi_h),
\quad\forall\varphi_h\in V_h
\end{equation} 
we get
\begin{align*}
  \left(\cfrac{{\partial} E_{v}}{\partial t},E_{v}\right)+\left({\nabla} E_{u},{\nabla} \cfrac{{\partial} E_{u}}{\partial t}\right)=\left(\cfrac{{\partial}\tilde{v}_{{\tau} h}}{\partial t}-f,E_{v}-{\Pi}_{h} E_{v}\right)+\left({\nabla} \tilde{u}_{{\tau}h},{\nabla}\left(E_{v}-\tilde I_{h} E_{v}\right)\right)\\
  + \left(p_n \left(A_h {\partial}_{n}^{2} u_{h}-{\partial}_{n}^{2}f_h\right),\tilde I_{h} E_{v}\right)-\left(p_n {\nabla}{\partial}_{n}^{2} v_{h},{\nabla}E_{u}\right)
  +\left(\tilde{f}_\tau-f,\tilde I_{h} E_{v}\right).
\end{align*}
Note that equation similar to (\ref{errode}) also holds for $t \in [t_0, t_{1}]$
\begin{align} 
\label{errorode1st}{b}({\partial}_{t}E,{\Phi})+(\mathcal{A}{\nabla} E,{\nabla} {\Phi})
  &={b} \left(\cfrac{{\partial} \tilde{U}_{{\tau} h}}{\partial t}-F,{\Phi}-{\Phi}_{h}\right)+\left(\mathcal{A}{\nabla} \tilde{U}_{{\tau} h},{\nabla} ({\Phi}-{\Phi}_{h})\right)\\
\notag  &+\left(p_1 \mathcal{A}{\nabla} {\partial}_{1}^{2} U_{h},{\nabla} {\Phi}_{h}\right)
  + {b}\left( \tilde{F}_\tau -F - p_1 \partial^2_1 F, \Phi_h\right). 
\end{align}
That follows from the definition of the piecewise quadratic reconstruction $\tilde{u}_{h\tau} (t)$ for $t \in [t_0, t_{1}]$. Integrating (\ref{errode}) and (\ref{errorode1st}) in time from 0 to some $t^{\ast}\geq t_1$ yields
\begin{align}
  \notag  & {\frac{1}{2}} \left(|E_{u}|^{2}_{H^1(\Omega)}+\| E_{v}\|_{L^2(\Omega)}^{2}\right)(t^{{\ast}}) \\
  &=
  {\frac{1}{2}} \left(|E_{u}|^{2}_{H^1(\Omega)}+\| E_{v}\|_{L^2(\Omega)}^{2}\right)(0)
  \notag\\
  &+ \int_{0}^{t^{{\ast}}}\left(\cfrac{{\partial\tilde{v}_{{\tau} h}}}{\partial t}-f, E_{v}-\tilde I_{h} E_{v}\right) d t
  +\int_{0}^{t^{{\ast}}}\left({\nabla} \tilde{u}_{{\tau} h},{\nabla} ( E_{v}-\tilde I_{h}  E_{v})\right) d t
 \notag \\
  &+\int_{t_1}^{t^{{\ast}}}\left[\left(p_n \left(A_h {\partial}_{n}^{2} u_{h}-{\partial}_{n}^{2}f_h\right),\tilde I_{h} E_{v}\right)-\left(p_n {\nabla}{\partial}_{n}^{2} v_{h},{\nabla}E_{u}\right)
  +\left(\tilde{f}_\tau-f,\tilde I_{h} E_{v}\right)\right] d t \notag\\
  &+\int_{0}^{t_1}\left[\left(p_1 \left(A_h {\partial}_{1}^{2} u_{h}-{\partial}_{1}^{2}f_h\right),\tilde I_{h} E_{v}\right)-\left(p_1 {\nabla}{\partial}_{1}^{2} v_{h},{\nabla}E_{u}\right)
  +\left(\tilde{f}_\tau-f,\tilde I_{h} E_{v}\right)\right] d t 
  \notag\\
  &\hspace{1cm}:=I+I I+I I I+IV .\notag
  \\
    \label{4newterms} 
  &
\end{align}
Let
\begin{equation*}
 Z (t) = \sqrt{| E_u |_{H^1(\Omega)}^2 + \| E_v \|_{L^2(\Omega)}^2}
\end{equation*}
and assume that $t^{\ast}$ is the point in time where $Z$ attains its maximum and $t^{\ast} \in (t_n, t_{n + 1}]$
for some $n$. For the first and second terms in (\ref{4newterms}) we have 

\begin{align*}
  I+II
  &\leq C_{1}\Biggl[\sum_{K{\in}\mathcal{T}_h}h_{K}^{2}\left\Vert\cfrac{{\partial}\tilde{v}_{h{\tau}}}{\partial t}-{\Delta}\tilde{u}_{h{\tau}}-f\right\Vert_{L^2( K)}^{2}\\
 &~~~~~~~~~~~~~~~~~~~~~~~~~~~~+\sum_{
E \in \mathcal{E}_h}h_{E}\left\|[n{\cdot}{\nabla}\tilde{u}_{h{\tau}}]\right\|_{L^2(E)}^{2}\Biggr]^{1/2}(t^{{\ast}})|E_{u}|_{H^1(\Omega)}(t^{{\ast}}) \\
  &+C_{1}\Biggl[\sum_{K{\in}\mathcal{T}_h}h_{K}^{2}\left\Vert\cfrac{{\partial}\tilde{v}_{h{\tau}}}{\partial t}-{\Delta}\tilde{u}_{h{\tau}}-f\right\Vert_{L^2(K)}^{2}\\
  &~~~~~~~~~~~~~~~~~~~~~~~~~~~~~~+\sum_{
E \in \mathcal{E}_h}h_{E}\left\|[n{\cdot}{\nabla}\tilde{u}_{h{\tau}}]\right\|_{L^2(E)}^{2}\Biggl]^{1/2}(0)|E_{u}|_{H^1(\Omega)}(0)\\
  &
  +C_{2}\sum_{m=1}^{n}\frac{{\tau}_{m-1}}{2}\left[\sum_{K{\in}\mathcal{T}_h}h_{K}^{2}\left\|{\partial}_{m}^{2}v_{h}-{\partial}_{m-1}^{2}v_{h}\right\|_{L^2( K)}^{2}\right]^{1/2}|E_{u}|_{H^1(\Omega)}(t_{m})\\
  &
  +C_{3}\sum_{m=0}^{n}\int_{t_{m}}^{min(t_{m+1},t^{{\ast}})}\Biggl[\sum_{K{\in}\mathcal{T}_h}h_{K}^{2}\left\Vert\cfrac{{\partial}^{2}\tilde{v}_{h{\tau}}}{\partial t^2}-{\Delta}\cfrac{{\partial}\tilde{u}_{{\tau}h}}{\partial t}-\cfrac{\partial f}{\partial t}\right\Vert_{L^2(K)}^{2}\\
  &~~~~~~~~~~~~~~~~~~~~~~~~~~~~+\sum_{
E \in \mathcal{E}_h}h_{E}\left\Vert\left[n{\cdot}{\nabla}\cfrac{{\partial}\tilde{u}_{{\tau}h}}{\partial t}\right]\right\Vert_{L^2(E)}^{2}\Biggr]^{1/2 }(t)|E_{u}|_{H^1(\Omega)}(t)dt.
\end{align*}
 
Indeed, it follows from integration by parts with respect to time, see the proof of Theorem 3.2 in \cite{our1}. The third and the fourth term in
(\ref{4newterms}) are responsible for the time estimator. We neglect the error at first 3 time steps, in other words we forget about the fourth term in
(\ref{4newterms}) and do not take into account terms at time steps before $t_3$. Thus the third term can be now written as

\begin{align}\label{thirdterm}
  \notag III  &= \sum_{n=3}^{N-1} \int_{t_m}^{\min (t_{m + 1}, t^{\ast})} \Biggl[\left(p_m \left(
  \partial_m^4 u_h + \tau_m  \sum_{k = m - 2}^m \gamma_k \left( \partial_k^2 
  \dot{f}_h - A_h \partial_k^2 v_h \right) \right), \tilde I_h E_v\right) \\
   &- (p_m \nabla \partial_m^2 v_h, \nabla E_u) + (\tilde{f}_{\tau} - f, \tilde I_h E_v)\Biggr] dt.
\end{align}

We now observe for $t \in [t_m,t_{m + 1}]$, $m = 1, \ldots, N-1$
\[ E_v = \frac{\partial_{} E_u}{\partial t} - p_{m} \partial_m^2 v_h \]
and integrate by parts in time in the terms in (\ref{thirdterm}) involving
$A_h \partial_k^2 v_h$
\begin{align*}
 & \int_{t_m}^{\min (t_{m + 1}, t^{\ast})} \tau_m p_m \left(A_h \partial_m^2 v_h,
  \tilde I_h E_v\right) {dt} \\
  &= \int_{t_m}^{\min (t_{m + 1}, t^{\ast})} \tau_m p_m \left(\nabla
  \partial_m^2 v_h, \nabla \tilde I_h \left( \frac{\partial_{} E_u}{\partial t} -
  p_{m} \partial_m^2 v_h \right)\right) {dt} \\
 &= \left[\tau_m p_m (\nabla \partial_k^2
  v_h, \nabla \tilde I_h E_u)\right]_{t_m}^{\min (t_{m + 1}, t^{\ast})} -  \int_{t_m}^{\min (t_{m + 1}, t^{\ast})} \tau_m p_m'  \left(\nabla \partial_m^2 v_h, \nabla \tilde I_h E_u\right) dt
  \\&- \int_{t_m}^{\min (t_{m + 1}, t^{\ast})} \tau_m p_m^2  \left(\nabla \partial_m^2
  v_h, \nabla \partial_m^2 v_h\right) dt.
\end{align*}
The first two terms above will result in a contribution to the error estimator
$\sim \tau_n^2  \left| \partial_n^2 v_h \right|_{H^1(\Omega)}$ (note that \ $\tau_n p_{n}'
\sim \tau_n^2$) while the remaining term is of a higher order. Thus the order of term h.o.t in error estimator (\ref{estf5}) is 
\begin{equation*}
    h.o.t.\sim \sum_{n=1}^{N-1} \left(\tau_n^4  \left\Vert \partial_n^2 \dot{f}_h \right\Vert_{L^2(\Omega)} +\tau_n^6  \left| \partial_n^2 v_h \right|_{H^1(\Omega)}\right).
  \end{equation*}
Thus noting the fact
$$
\int_{t_m}^{t_{m + 1}} |p_m|{dt} \leq \frac{1}{12}{\tau}_{m}^{3}+\frac{1}{8}{\tau}_{m-1}{\tau}_{m}^2
$$
we obtain (\ref{in5}).
\quad  
\end{proof}

Finally we show the upper bound for $5$-point time estimator at least for sufficiently smooth solutions, quasi-uniform meshes in space and uniform meshes in time.
\begin{theorem}
Let u be the solution of wave equation (\ref{wave}) and 
$\displaystyle\cfrac{\partial^{3}u}{\partial t^{3}}(0)\in {H^1(\Omega)}$, $\displaystyle\cfrac{\partial^{2}u}{\partial t^{2}}(0) \in {H^2(\Omega)}$, $\displaystyle\cfrac{\partial^{2}f}{\partial t^{2}}(t)\in L^{\infty}(0,T;{L^2(\Omega)})$, $\displaystyle\cfrac{\partial^{3}f}{\partial t^{3}}(t)\in L^{2}(0,T;{L^2(\Omega)})$. Suppose that mesh $\Th$ is quasi-uniform and the mesh in time is uniform ($t_k=k\tau$). Then, the 5-point time error estimator $\hat{\eta}_T(t_k)$ defined by (\ref{in5}) is of order $\tau^2$, i.e. 
\begin{equation*}
\hat{\eta}_T(t_k)\leq C \tau^2.
\end{equation*}
with a positive constant $C$ depending only on $u$, $f$, and the mesh regularity.
\end{theorem}
\begin{proof}
The result follows from Theorem \ref{optimality} by using (\ref{Newmhd}) and Lemma \ref{lemtech1}
\begin{align*}
\left\|\partial_n^4 u_h \right\|_{L^2(\Omega) }= \left\|\hat{\partial}^2_n  \bar{f}_h - A_h  \hat{\partial}^2_n  \bar{u}_h\right\|_{L^2(\Omega) } =  \left\|\sum_{k = n - 2}^n \alpha_k\left({\partial}^2_k  {f}_h - A_h  {\partial}^2_k  {u}_h\right)\right\|_{L^2(\Omega)}. 
\end{align*}
\end{proof}
\begin{remark}\label{intImp2}
Note, that as in the case for 3-point error estimator, the approximation of initial conditions and right-hand-side function is crucial for optimality of our time and space error estimators. 
\end{remark}

\section{Numerical results}\label{section4}
\subsection{A toy model: a second order ordinary differential equation}

Let us consider first the following ordinary differential equation
\begin{equation}
\begin{cases}
\cfrac{d^{2}u(t)}{dt^{2}}+Au(t)=f(t) ,&t\in\left[ 0;T\right]\\
u(0)=u_0 ,&\\
\cfrac{du}{dt}(0)=v_0&
\end{cases}
\label{ODE}
\end{equation}
with a constant $A>0$. This problem serves as simplification of the wave equation in which we get rid of the space variable. The Newmark scheme reduces in this case to
\begin{align}
\frac{u^{n+1}-u^{n}}{\tau_n}-\frac{u^{n}-u^{n-1}}{\tau_{n-1}}&+A\frac{\tau_n(u^{n+1}+u^{n})+\tau_{n-1}(u^{n}+u^{n-1})}{4}=
\notag\\
&=\frac{\tau_n (f^{n+1}+f^n)+\tau_{n-1}(f^n+f^{n-1})}{4},~1\leq n\leq N-1
\label{schN}\\
\frac{u^1-u^0}{\tau_0}&=v_0-\frac{\tau_0}{4}A(u^1+u^0)+\frac{\tau_0}{4}(f^1+f^0),
\notag\\
u^0&=u_0\notag
\end{align}
and the error become $e=\displaystyle\max_{0 \leq n \leq N}\left( \left\vert v^{n}-{u}'(t_{n})\right\vert ^{2}+A\left\vert u^{n}-u(t_{n})\right\vert ^{2}\right) ^{{1}/{2}}$. 3-point and 5-point \textit{a posteriori} error estimates $\forall n:~0\leq n \leq N$ simplify to this form:
\begin{align}\label{errest}
e \leq \sum_{k=0}^{n-1}\tau_k\eta_{T}(t_k)&= \tau_0 \left(\frac{5}{12}\tau_{0}^2+\frac{1}{2}\tau_{0}\tau_{1}\right) \sqrt{A(\partial_1^2 v)^2 +  (\partial_1^2f-A\partial_1^2 u)^2}\\
\notag   &+\sum_{k=1}^{n-1}\tau_k \left(\frac{1}{12}\tau_{k}^2+\frac{1}{8}\tau_{k-1}\tau_{k}\right) \sqrt{A(\partial_k^2 v)^2 +  (\partial_k^2f-A\partial_k^2 u)^2},  \\
e \leq \sum_{k=3}^{n-1}\hat{\eta}_{T}(t_k)&= \sum_{k=3}^{n-1}\tau_k \left(\frac{1}{12}\tau_{k}^2+\frac{1}{8}\tau_{k-1}\tau_{k}\right) \sqrt{A(\partial_k^2 v)^2 +  (\partial_k^4
   u)^2}
\end{align}

We define the following effectivity indexes  in order to measure the quality of our estimators $\eta_T$ and $\hat{\eta}_T$
\begin{equation*}
ei_T=\frac{\eta_T}{e},~~\hat{ei}_T=\frac{\hat{\eta}_T}{e}.
\end{equation*}
We present in Table \ref{tab:ode} the results for equation (\ref{ODE}) setting $f=0$, the exact solution $u=cos(\sqrt{A}t)$, final time $T=1$, and using constant time steps $\tau=\displaystyle {T}/{N}$. We observe that 3-point and 5-point estimators are divided by about 100 when the time step $\tau$ is divided by 10. The true error $e$ also behaves as $O(\tau^2)$ and hence both time error estimators behave as the true error.

\begin{table}
\begin{center}
\begin{tabular}{llrrrrr}
\hline\noalign{\smallskip}
$A$ & $N$ & $\eta_{T}$ & $\hat{\eta}_{T}$& $e$ & $ei_{T}$ & $\hat{ei}_{T}$\\
\noalign{\smallskip}\hline\noalign{\smallskip}
100 & 100 & .21 & .203 & .085 & 2.47 & 2.39 \\
100 & 1000 & .0021 & .0021 & 8.34e-04 & 2.5 & 2.49\\
100 & 10000 & 2.08e-05 & 2.08e-05 & 8.35e-06 & 2.5 & 2.5 \\
\noalign{\smallskip}\hline\noalign{\smallskip}
1000 & 100 & 20.51 & 19.47 & 8.35 & 2.46 & 2.33 \\
1000 & 1000 & .209 & .208 & .084 & 2.5 & 2.49\\
1000 & 10000 & .0021 & .0021 & 8.33e-04 & 2.5 & 2.5\\
\noalign{\smallskip}\hline\noalign{\smallskip}
10000 & 100 & 1.68e+03 & 1.4e+03 & 200 & 8.38 & 6.98\\
10000 & 1000 & 20.8 & 20.7 & 8.34 & 2.5 & 2.49\\
10000 & 10000 & .208 & .208 & .083 & 2.5 & 2.5\\
\noalign{\smallskip}\hline
\end{tabular}
\end{center}
\caption{Effective indices for constant time steps and $f=0$.}
\label{tab:ode}
\end{table}

In order to check behaviour of time error estimators for variable time step (see Table \ref{tab:ode2}) we take the previous example with time step $\forall n:~0\leq n \leq N$
\begin{equation}\label{tau10}
\tau_n=\begin{cases}
0.1\tau_{\ast} ,&mod(n,2)=0\\
\tau_{\ast} ,&mod(n,2)=1
\end{cases}
\end{equation}
where $\tau_{\ast}$ is a given fixed value.  As in the case of constant time step we have the equivalence between the true error and both estimated errors. We have plotted on Fig \ref{fig:toyIndicators} evolution in time of the values $\sum_{k=0}^{n-1}{\eta}_{T}(t_k)$ and $\sum_{k=3}^{n-1}\hat{\eta}_{T}(t_k)$ compared to $e$. 

Table {\ref{tab:ode3}} contains the results for even more non-uniform time step $\forall n:~0\leq n \leq N$
\begin{equation}\label{tau100}
\tau_n=\begin{cases}
0.01\tau_{\ast} ,&mod(n,2)=0\\
\tau_{\ast} ,&mod(n,2)=1
\end{cases}
\end{equation}
on otherwise the same test case. 
Note that in case when $A=100$ and $N=19800$ 5-points error estimator $\hat{\eta}_{T}$ blows up, while 3-point estimator behaves as the true error. This effect is consistent with Theorem {\ref{lemest3}}. Indeed, the constants in the bounds of this Lemma may depend on the meshes regularity in time.

Our conclusion is thus that for toy model classic and alternative \textit{a posteriori} error estimators are sharp on both constant and variable time grids. 
\begin{figure}[h!]
    \centering
    \includegraphics[width=.85\textwidth]{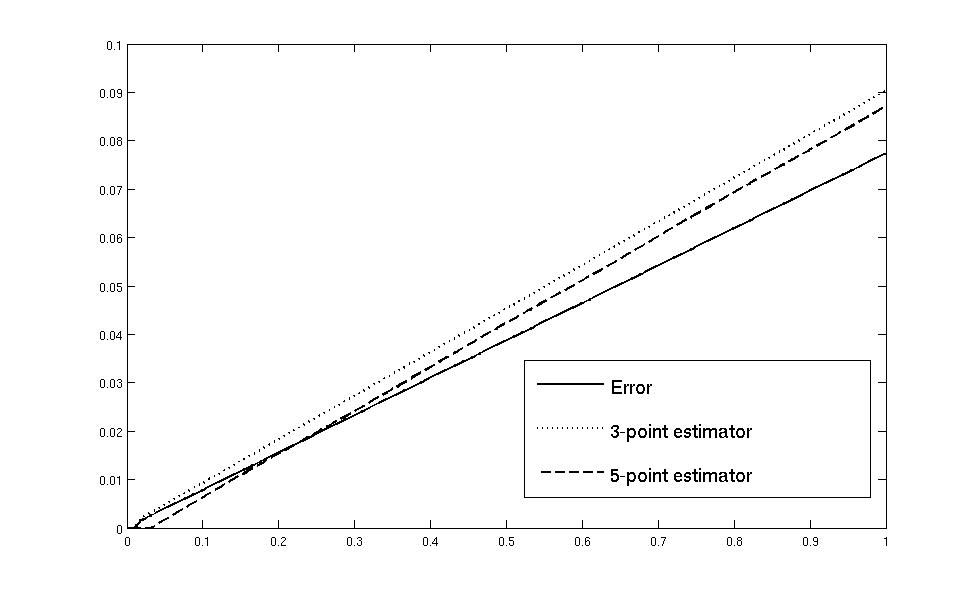}
    \caption{Evolution in time of 3-point and 5-point time estimators for variable time step (\ref{tau10}), $A=100$, $N=180$, $T=1$}
    \label{fig:toyIndicators}
\end{figure} 
\begin{table}[h!]
\begin{center}
\begin{tabular}{llrrrrr}
\hline\noalign{\smallskip}
$A$ & $N$ & $\eta_{T}$ & $\hat{\eta}_{T}$& $e$ & $ei_{T}$ & $\hat{ei}_{T}$\\\noalign{\smallskip}\hline\noalign{\smallskip}
100 & 180 & .09 & .087 & .077 & 1.17 & 1.13 \\
100 & 1816 & 8.85e-04 & 8.82 e-04 & 7.59e-04 & 1.17 & 1.16\\
100 & 18180 & 8.83e-06 & 8.83e-06 & 7.6e-06 & 1.16 & 1.16 \\
\noalign{\smallskip}\hline\noalign{\smallskip}
1000 & 180 & 8.91 & 8.52 & 7.6 & 1.17 & 1.13 \\
1000 & 1816 & .089 & .088 & .076 & 1.17 & 1.16\\
1000 & 18180 & 8.84e-04 & 8.83e-04 & 7.59e-04 & 1.16 & 1.16\\
\noalign{\smallskip}\hline\noalign{\smallskip}
10000 & 180 & 802.84 & 725.1 & 200 & 4.01 &3.63\\
10000 & 1816 & 8.84 & 8.8 & 7.58 & 1.17 & 1.16\\
10000 & 18180 & .088 & .088 & .076 & 1.16 & 1.16\\
\noalign{\smallskip}\hline
\end{tabular}
\end{center}
\caption{Effective indices for variable time step (\ref{tau10}) and $f=0$.}
\label{tab:ode2}
\end{table}
\begin{table}[h!]
\begin{center}
\begin{tabular}{llrrrrr}
\hline\noalign{\smallskip}
$A$ & $N$ & $\eta_{T}$ & $\hat{\eta}_{T}$& $e$ & $ei_{T}$ & $\hat{ei}_{T}$\\\noalign{\smallskip}\hline\noalign{\smallskip}
100 & 196 & .086 & .083 & .084 & 1.02 & 0.98 \\
100 & 1978 & 8.39e-04 & 8.36 e-04 & 8.26e-04 & 1.02 & 1.01\\
100 & 19800 & 8.38e-06 & 1.82e-05 & 8.1e-06 & 1.03 & 2.24 \\
\noalign{\smallskip}\hline\noalign{\smallskip}
1000 & 196 & 8.47 & 8.1 & 8.26 & 1.02 & 0.98 \\
1000 & 1978 & .083 & .084 & .0827 & 1.02 & 1.01\\
1000 & 19800 & 8.37e-04 & 8.37e-04 & 8.26e-04 & 1.01 & 1.01\\
\noalign{\smallskip}\hline\noalign{\smallskip}
10000 & 196 & 764.2 & 691.7 & 200 & 3.82 &3.46\\
10000 & 1978 & 8.39 & 8.35 & 8.25 & 1.02 & 1.01\\
10000 & 19800 & .084 & .084 & .083 & 1.01 & 1.01\\
\noalign{\smallskip}\hline
\end{tabular}
\end{center}
\caption{Effective indices for variable time step (\ref{tau100}) and $f=0$.}
\label{tab:ode3}
\end{table}

\subsection{The error estimator for the wave equation on Delaunay mesh}
We now report numerical results for initial boundary-value problem for wave equation with non-uniform time steps when using 3-point time error estimator (\ref{in}, \ref{in0step}) and 5-point time error estimator(\ref{in5}).
We compute two parts of the space estimator (\ref{space}) in practice as follows:
\begin{align}
 \notag \eta_S^{(1)} (t_N) = \max_{1 \leq n \leq N-1} \left[ \sum_{K \in \mathcal{T}_h}
   h_K^2  \left\Vert  \partial_n v_h-{f}^n_{h}
   \right\Vert_{L^2(K)}^2\right. +& \left.\sum_{
E \in \mathcal{E}_h}h_{E} \|[n \cdot \nabla {u}^n_{h}]\|_{L^2(E)}^2 \right]^{1/2},
\\
\notag  \eta_S^{(2)} (t_N)  = \sum_{n = 1}^{N-1} \tau_n \left[ \sum_{K \in \mathcal{T}_h} h_K^2
  \left\Vert  \partial_n^2 v_h - \partial_n f_h\right\Vert_{L^2(K)}^2 \right.&\\
  + &\left.\sum_{
E \in \mathcal{E}_h}h_{E} \left\Vert\left[n \cdot
  \nabla  \partial_n u_h\right]\right\Vert_{L^2(E)}^2
  \right]^{1/2}. \label{etas1}   
\end{align}

The quality of our error estimators in space and time is determined by following effectivity indices:
\begin{equation*}
ei=\frac{\eta_T+\eta_S}{e},~~\hat{ei}=\frac{\hat{\eta}_T+\eta_S}{e}\end{equation*}
where the first index contains 3-point time error estimator and space estimator, while the last measures the grade of 5-point time error estimator and space estimator. The true error is
\begin{equation*}
e=\max_{0 \leqslant n \leqslant N}\left(\left\Vert v^{n}_h-\cfrac{\partial u}{\partial t}(t_{n})\right\Vert_{L^2(\Omega)} ^{2}+\left\vert u^{n}_h-u(t_{n})\right\vert^{2}_{H^1(\Omega)}\right) ^{{1}/{2}}
\end{equation*}
Consider the problem (\ref{wave}) with $\Omega=(0,1)\times(0,1),~ T=1$ and the exact solution $u$ given by
\begin{equation}\label{exponenta}
 u(x,y,t)=e^{-100r^2(x,y,t)},
\end{equation}
where
\begin{equation}
r^2(x,y,t)=(x-0.3-0.4t^2)^2+(y-0.3-0.4t^2)^2.
\end{equation}
Thus, $u$ is a Gaussian function, whose center moves from point $(0.3,0.3)$ at $t=0$ to point $(0.7,0.7)$ at $t=1$. The transport velocity $0.8t(1,1)^T$ is peaking at $t=1$. We choose non-uniform time step as $\tau_n=\tau_0\displaystyle\cfrac{1}{{t_n}^{1/2}}$ for $n=1,\ldots,N-1$.
We interpolate initial conditions and right-hand-side function with the orthogonal projections as it noted in Remarks \ref{intImp}, \ref{intImp2}
\begin{equation}\label{Projinit}
u^0_h = \Pi_h u^0,~v^0_h = \Pi_h v^0,~f^n_h= P_h f^n,~0\leq n \leq N.
\end{equation} 
Unstructured  Delaunay meshes in space are used in all the experiments. Numerical results are reported in Table \ref{tab:wave3}. Note that this case is chosen so that the non-uniform time step is required, see Fig.\ref{Example1}.

\begin{figure}
\centering \includegraphics[height=0.4\hsize]{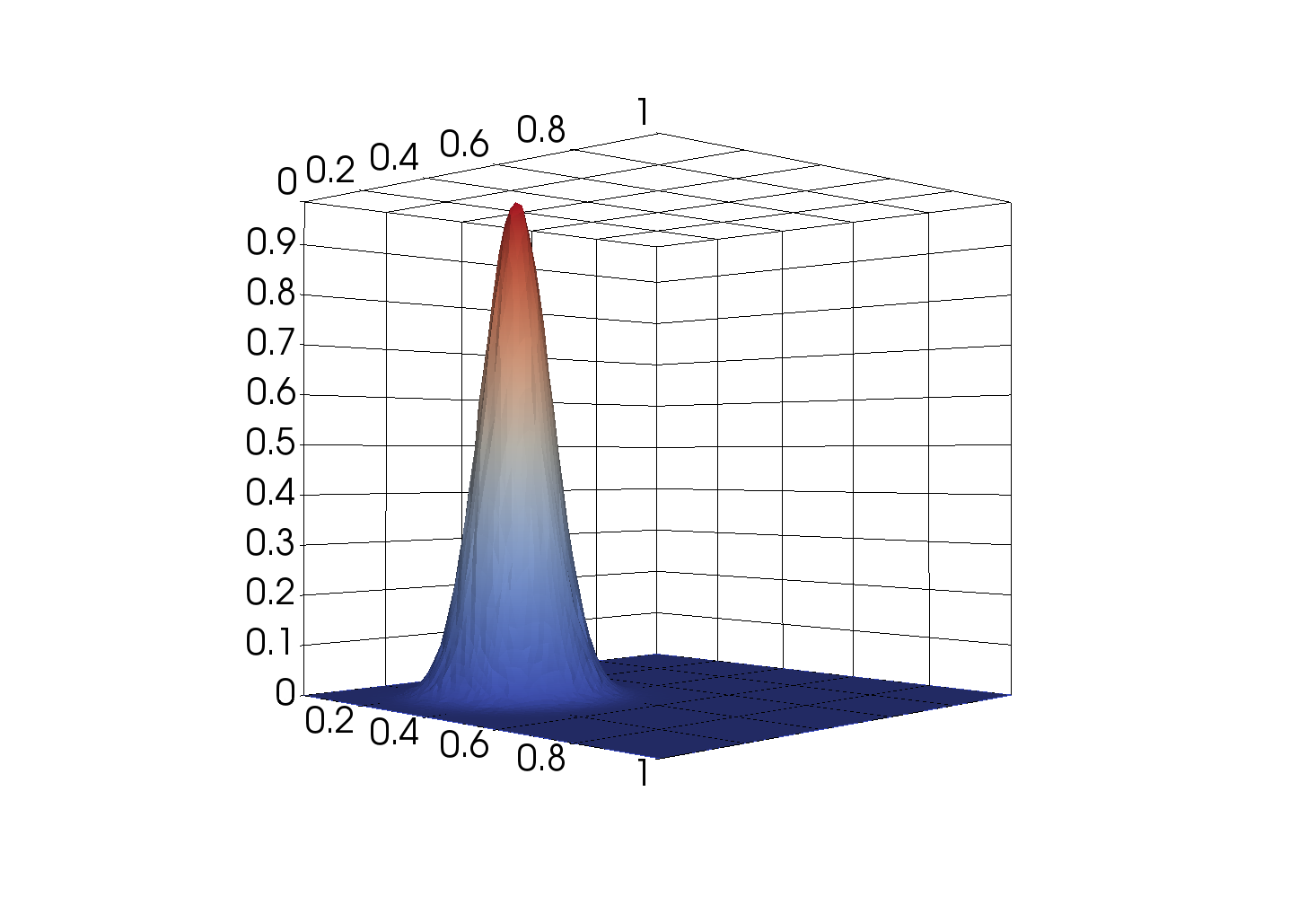}
\includegraphics[height=0.4\hsize]{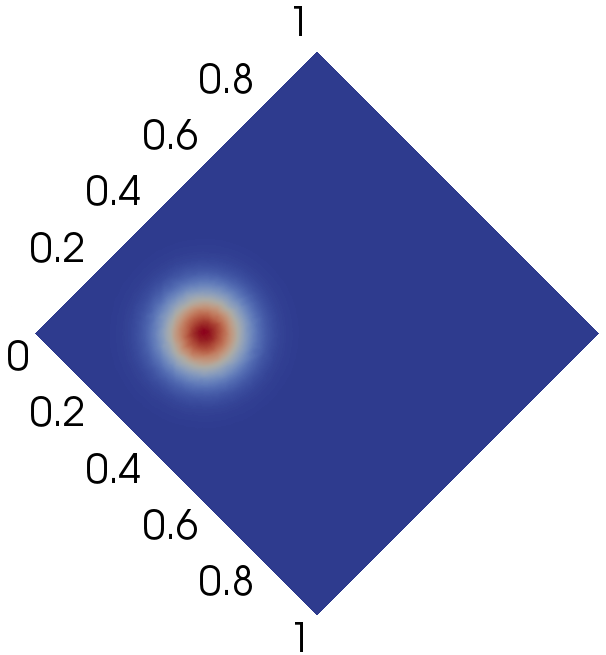}
\includegraphics[height=0.4\hsize]{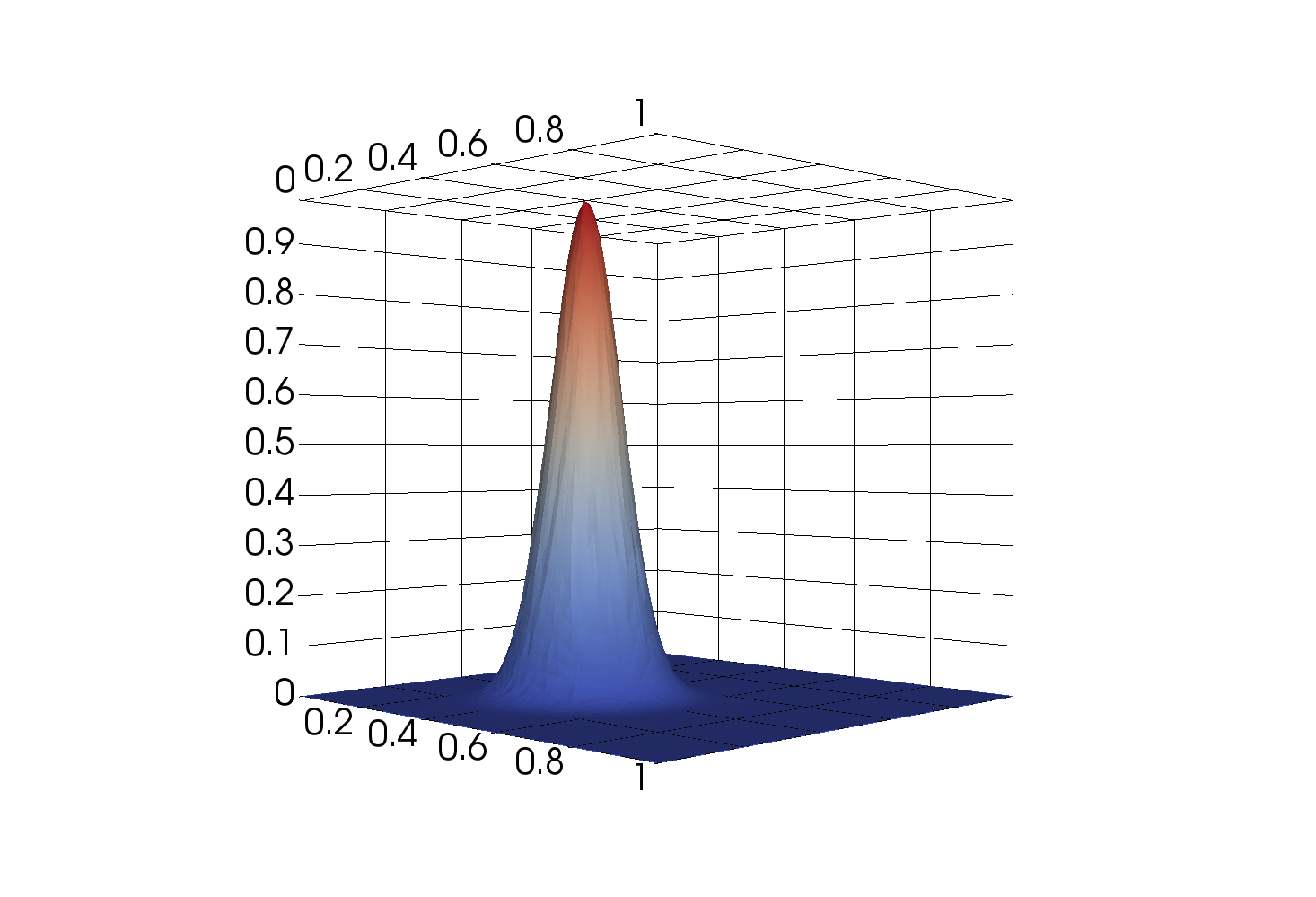}
\includegraphics[height=0.4\hsize]{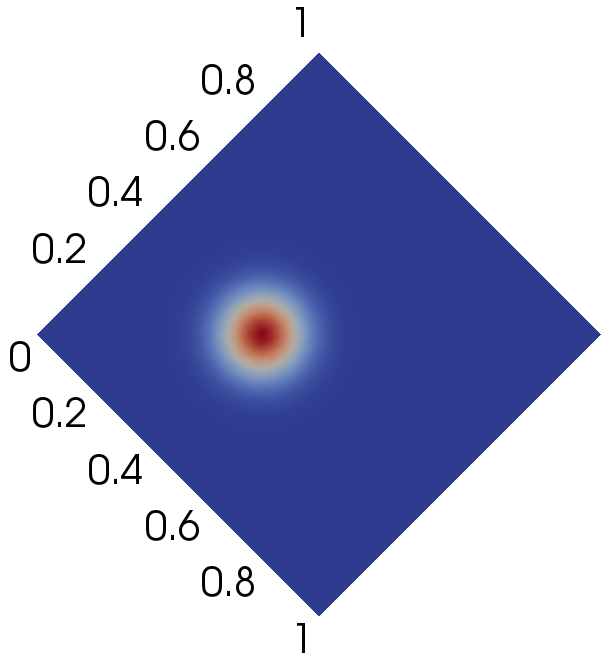}
\includegraphics[height=0.4\hsize]{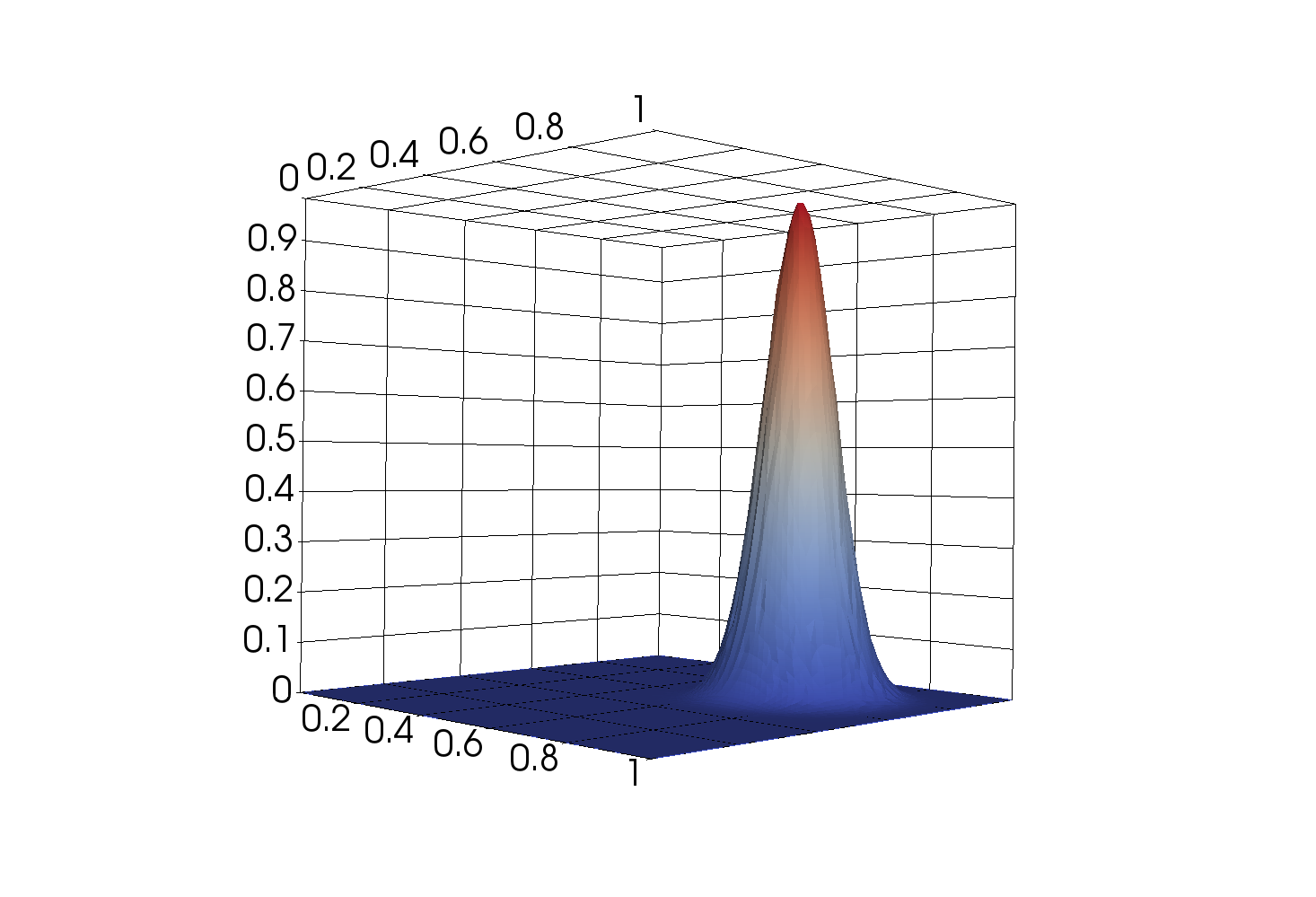}
\includegraphics[height=0.4\hsize]{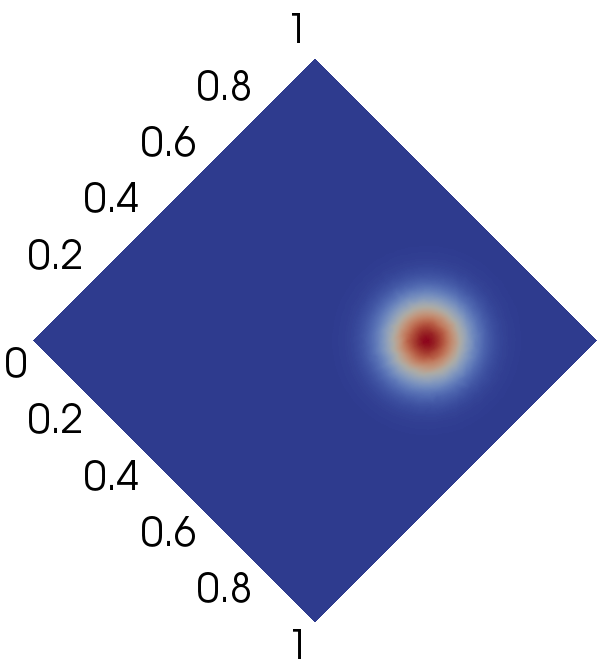}
\caption{{Solution (\ref{exponenta}). From up to down: time $t=0,~0.5,~1$.}}
\label{Example1}
\end{figure}

\begin{table}[h!]
\caption{Results for  case (\ref{exponenta}), non-uniform time step.}
\label{tab:wave3}
\begin{tabular}{llrrrrrrrr}
\hline\noalign{\smallskip}
$h$ & $\tau_0$ & $ei$ & $\hat{ei}$ & $\eta_T$ & $\hat{\eta}_T$ & $\eta_S$ &$\tau_F$ &$N_{ts}$ & e\\
\noalign{\smallskip}\hline\noalign{\smallskip}
.05 & .01 & 4.85 & 4.83 & .096 & .088 & 2.55 & .0063 & 105 & .58 \\
.025  & .0071 & 5.39 & 5.38 & .054 & .051 & 1.39 & .0045 & 149 & .27\\
.0125  & .005 & 5.94 & 5.93 & .028 & .026 & .72 & .0032 & 210 & .13\\
.00625 & .0035 & 5.94 & 5.94 & .014 & .013 & .36 & .0022 & 297 & .065\\
.003125 & .0025 & 5.94 & 5.94 & .0067 & .0065 & .18 & .0016 & 421 & .032\\
\noalign{\smallskip}\hline
\end{tabular}
\end{table}

Referring to Table \ref{tab:wave3}, we observe that when setting initial time step as $\tau^2\sim O(h)$ the error is divided by 2 each time $h$ is divided by 2, consistent with $e\sim O(\tau^2+h)$. The space error estimator and the two time error estimators behave similarly and thus provide a good representation of the true error. Both effectivity indices tend to a constant value. 

We therefore conclude that our space and time error estimators are sharp in the regime of non-uniform time steps and Delaunay space meshes. They separate well the two sources of the error and can be thus used for the mesh adaptation in space and time. In particularly, 3-point and 5-point time estimators become more and more close to each other when $h$ and $\tau$ tend to 0.

\bibliographystyle{acm}
\bibliography{apost}
-
\end{document}